\documentclass[11pt, oneside]{article}   

\usepackage{geometry}    
\usepackage{graphicx}
\usepackage{amssymb}
\usepackage{amsthm}
\usepackage{rotating}
\newtheorem{thm}{Theorem}[section]

\newtheorem{cor}[thm]{Corollary}
\newtheorem{lem}[thm]{Lemma}
\theoremstyle{definition}

\def\Core{{\mathrm{Core}}}

\def\GAP{{\sf GAP}}

\def\C{{\mathbb{C}}}
\def\Irr{\mbox{\rm Irr}}

\title{Subgroups of arbitrary even ordinary depth}
\author{Hayder Abbas Janabi, Thomas Breuer, and Erzs\'ebet Horv\'ath}
\date{June 16, 2020}  

\begin{document}

\maketitle

\begin{abstract}
\noindent
We show that for each positive integer $n$,
there are a group $G$ and a subgroup $H$ such that the ordinary depth
$d(H, G)$ is $2n$.
This solves the open problem posed by Lars Kadison
whether even ordinary depth larger than $6$ can occur.

\vspace*{1mm}

\noindent
{\bf Mathematics Subject Classification (2010):}
primary 20C15; secondary 20B35, 05C12, 05C76.

\vspace*{1mm}

\noindent
{\bf Keywords:} ordinary depth of a subgroup, distance of characters, Cartesian product of graphs, wreath product.
\end{abstract}

\section{Introduction}

The notion of  depth was originally defined for von-Neumann algebras,
see~\cite{GHJ}.
Later it was also defined for Hopf algebras, see~\cite{KN}.
For some recent results in this direction, see~\cite{HKY,K,KHSZ}.
In~\cite{KK} and later in~\cite{BKK}, the depth of semisimple algebra inclusions was studied, by Burciu, Kadison and K\"ulshammer.
 First results were considering the depth $2$ case, later it was generalized for arbitrary $n$.
In the case of group algebra inclusion ${\C}G\subseteq {\C}H$ it was shown that the depth is at most $2$ if and only if $H$ is normal in $G$, see~\cite{KK}.

Now let us remind the reader to the notion of \emph{ordinary depth} of a subgroup
 $H$ in the finite group $G$.
We say that the \emph{ depth of the group algebra inclusion } ${\C}H \subseteq {\C}G$ is $2n$
if ${\C}G\otimes_{{\C}H}\cdots\otimes_{{\C}H}{{\C}G}$ ($n+1$-times ${\C}G$)
 is isomorphic to a direct summand of
$\oplus_{i=1}^a{\C}G\otimes_{{\C}H}\cdots \otimes_{{\C}H}{\C}G$ ($n$ times ${\C}G$)
as ${\C}G-{\C}H$-bimodules  (or  equivalently as ${\C}H-{\C}G$-bimodules) for some positive integer $a$.

Furthermore, ${\C}H$ is said to have depth $2n+1$ in ${\C}G$ if the same assertion holds for  ${\C}H-{\C}H$-bimodules. Finally ${\C}H$ has depth $1$ in ${\C}G$ if
${\C}G$ is isomorphic to a direct summand of $\oplus_{i=1}^a{\C}H$ as
 ${\C}H-{\C}H$ bimodules.
The minimal depth of group algebra inclusion ${\C}H\subseteq {\C}G$ is called
(minimal) \emph{ordinary depth of $H$ in $G$}, which we denote by $d(H,G)$.
 If one considers  group algebras over a field of characterstic $p$ then one gets in a similar way the notion of modular depth and it does not depend on the field,  only on the characteristic, as it is shown in \cite[Remark 4.5]{BDK} by Boltje, Danz and K\"ulshammer.

  The ordinary depth can be obtained from
the so called \emph{inclusion  or Frobenius matrix} $M$. If $\chi_1,\ldots ,\chi_s$ are
all irreducible characters of $G$ and $\psi_1,\ldots ,\psi_r$ are all irreducible characters of $H$, then $m_{i,j}:=(\psi_i^G,\chi_j)$.
The "powers" of $M$ are defined by $M^{2l}:=M^{2l-1}M^T$ and
 $M^{2l+1}:=M^{2l}M$.
The ordinary depth $d(H,G)$ can be obtained as the smallest integer $n$
such that $M^{n+1}\leq aM^{n-1}$ for some positive integer $a$, where the inequality of matrices means that this inequality holds componentwise.

The  results on characters in \cite{BKK}   help
to determine $d(H,G)$. Two irreducible characters $\alpha,\beta \in \Irr(H)$ are called \emph{related},
$\alpha \sim_G \beta$, if they are constituents of $\chi_H$, for some
 $\chi \in \Irr(G)$. The \emph{distance} $d_G(\alpha,\beta)=m$
is the smallest integer $m$ such that there is a chain of irreducible
 characters of $H$ such that $\alpha=\psi_0\sim_G \psi_1\ldots \sim_G
\psi_m=\beta $. If there is no such chain then
$d_G(\alpha,\beta)=-\infty $ and if $\alpha=\beta $ then the distance is zero.
If $X$ is the set of irreducible constituents of $\chi_H$ then
we set
$m(\chi):= \max\{\min\{d_G(\alpha,\psi); \psi\in X\}; \alpha\in\Irr(H)\}$.
We will use the following result from \cite{BKK}.

\begin{thm}\cite[Thm~3.6, Thm~ 3.10]{BKK}
\begin{itemize}
\item[(i)] Let $m\geq 1$. Then $H$ has ordinary depth $\leq 2m+1$ in $G$
if and only if the distance between two irreducible characters of $H$ is at most $m$.
\item[ (ii)] Let $m\geq 2$. Then $H$ has ordinary depth $\leq 2m$ in $G$
if and only if $m(\chi)\leq m-1$ for all $\chi \in \Irr(G)$.
\end{itemize}
\end{thm}

Thus we have the following.

\begin{cor}\label{distance}
Let $H$ be a subgroup of a finite group $G$.
The \emph{ordinary depth} $d(H,G)$ is the minimal possible positive integer
which can be determined from the following upper bounds:
                              
\begin{itemize}
\item[(i)]
   For $m \geq 1$, $d(H,G) \leq 2m+1$
   if and only if the distance between two irreducible characters of $H$
   is at most $m$.
\item[(ii)]
   For $m \geq 2$, $d(H,G) \leq 2m$
   if and only if $m(\chi) \leq m-1$ for all $\chi \in \Irr(G)$.
\item[(iii)]
   $d(H,G) \leq 2$ if and only if $H$ is normal in $G$,
\item[(iv)]   $d(H,G)=1$ if and only if $G = H C_G(x)$ for all $x \in H$,
   see \cite[Thm. 1.7 ]{BK}.
\end{itemize}
\end{cor}

We will also use the following result from \cite{BKK}.

\begin{thm}\cite[Thm~6.9]{BKK}\label{inter}
Suppose that $H$ is a subgroup of a finite group $G$ and $N = \Core_G(H)$
is the intersection of $m$ conjugates of $H$.
Then $d(H,G) \leq 2m$.
If additionally $N \leq Z(G)$ holds then $d(H,G) \leq 2m-1$.
\end{thm}

In recent publications, several authors determined the ordinary depth of subgroups in some special series of groups,  e.g. $PSL(2,q)$, Suzuki groups,
Ree groups,  symmetric and alternating groups, see \cite{BKK}, \cite{Frit}, \cite{FKR}, \cite{HHP2015}, \cite{HHP2019}. In \cite{Danz}, twisted group algebra inclusions for symmetric and alternating groups are studied.

It is known that odd ordinary depth of a subgroup in a finite group
can be arbitrarily large:
It is shown in~\cite{BKK} that the depth of the symmetric group
$S_n$ in $S_{n+1}$ is $2n-1$.

Lars Kadison posed the following  open problem on his homepage, see~\cite{Khp}:
Are there subgroups of ordinary depth $2n$ where $n > 3$?

If one looks at the results of the above papers or the calculations presented
in ~\cite{Birmingham2017}, one has the impression that in most cases the depth of subgroups is odd. However still one can find examples of arbitrarily large even depth.
In our examples wreath products will  play an  important role.
There is another depth concept, the combinatorial depth, which is considered
in several investigations, see e.g. \cite{BDK}, \cite{HHP2015}, \cite{HHP2019}.
In this short note we will always consider ordinary depth, so in the following depth will always mean ordinary depth.

The main result of this paper is the following.

\begin{thm}\label{2n}  There exists a series of groups and subgroups $(G_n,H_n)$ such
that
$d(H_n,G_n)=2n$ for every positive integer $n$.
\end{thm}

\section{Constructing examples}

Examples of subgroups of depth $8$ had been constructed earlier
by the third author with the help of the {\GAP} system~\cite{GAP4},
see~\cite{Birmingham2017}.
Let us mention some of them:
\begin{itemize}
\item{}
    $d(A_{15}\cap (S_{12}\times S_3),A_{15})=8$,
\item{} $d((2^6:U_4(2)),O_8^{-}(2))=8$, and
\item{}
$d(((((C_2\wr C_2)\wr C_2)\wr C_2)\cap A_8\times A_8),((C_2\wr C_2)\wr C_2)\wr C_2))=8$.
\end{itemize}

It was shown already in \cite{BKK} that $d(D_8,S_4)=4$.
The first author found with GAP that\\
 $d(D_8\times S_4,S_4\wr C_2)=8$.
Continuing this process, we  obtained that\\
 $d((D_8\times S_4)\times (S_4\wr C_2),(S_4\wr C_2)\wr C_2))=16$ and in general

\begin{itemize}
\item{}
$G_0:=S_4$, $H_0:=D_8$,
\item{} $G_n:=G_{n-1}\wr C_2$,
$H_n:=H_{n-1}\times G_{n-1}<G_{n-1}\times G_{n-1}<G_n$,
\item{} $d(H_n,G_n)=2^{n+2}$.
\end{itemize}
 The idea of the proof is to use Thm.~\ref{inter} to prove that $d(H_n,G_n)\leq 2^{n+2}$. Then we show that  the depth cannot be at most  $2^{n+2}-1=2(2^{n+1}-1)+1$, since by Cor.~\ref{distance} then the distance of any two characters of $H_n$ would be at most $2^{n+1}-1$, however there are
irreducible characters of $H_n$ of distance exactly $2^{n+1}$.
The proof is a rather complicated induction, see \cite{JBH}.

We wanted to simplify the construction. Our aim was also to construct as depth more even numbers. We can generalize the first two steps of
the former construction in another way as follows:
\begin{itemize}
\item{} $d(D_8,S_4)=4,$
\item{} $d(D_8\times S_4,S_4\wr C_2)=8,$
\item{} $d(D_8\times S_4\times S_4,S_4\wr C_3)=12.$
\end{itemize}

In general, we have that

\begin{itemize}
\item{} $G_1:=S_4$,  $H_1:=D_8$,
\item{} $G_n:=G_1\wr C_n$,
\item{} $H_n:=H_1\times G_1^{n-1}<G_1^n<G_n$.
\end{itemize}

Then we have that $d(H_n,G_n)=4n$. The proof is again using Thm. \ref{inter}
to prove that $d(H_n,G_n)\leq 4n$. If $d(H_n,G_n)\leq 4n-1=2(2n-1)+1$, then
by Cor. \ref{distance}
any two irreducible characters of $H_n$ have distance at most  $2n-1$. However,  one can show that there exist irreducible characters of $H_n$ of distance $2n$.

If we want to get every even number then we can use a modified  construction.
We may take the Klein four group $V_4\triangleleft S_4$ instead of $D_8$
and get:

\begin{itemize}
\item{} $d(V_4,S_4)=2,$
\item{} $d(V_4\times S_4,S_4\wr C_2)=4,$
\item{} $d(V_4\times S_4\times S_4,S_4\wr C_3)=6.$
\end{itemize}

In general, we have a series  of groups and subgroups such that
$d(H_n, G_n) = 2n$ holds. The idea of the proof will be the same as before,
for the inequality  we will use again Thm.~\ref{inter}, and to prove that
it cannot be a strict inequality,
we find two irreducible characters of distance $n$ in $H_n$.
For that, we consider suitable characters of the base group of the wreath product and define a Cartesian product of graphs that encodes the relation $\sim$.

\section{Proof of Theorem~\ref{2n}}

Let $G$ be the symmetric group on four points,
and $N$ be its normal Klein four subgroup.
Set $G_1 = G$, $H_1 = N$.
Then $d(H_1, G_1) = 2$, by Cor.~\ref{distance}.
Define for $n \geq 2$
\begin{eqnarray*}
   \sigma_n & = & \prod_{j=1}^4 (j, j + 4, j+8, \cdots, j+4(n-1)), \\
   G_n      & = & \langle G, \sigma_n \rangle, \\
   H_n      & = & \langle N, G^{\sigma_n}, G^{\sigma_n^2}, \ldots,
                             G^{\sigma_n^{n-1}} \rangle.
\end{eqnarray*}

Let $C_n$ denote the cyclic group of order $n$.
Then $H_n < G_n \cong G \wr C_n$
and
\[
   H_n \cong N \times G^{n-1} \leq G^n < G \wr C_n.
\]
Let $N_n = \Core_{G_n}(H_n)$, the largest normal subgroup of $G_n$ that is
contained in $H_n$.
Then $N_1 = N$, and
\[
   N_n = \langle N, N^{\sigma_n}, \ldots, N^{\sigma_n^{n-1}} \rangle
       = \bigcap_{i=0}^{n-1} H_n^{\sigma_n^i}
\]
is an intersection of $n$ conjugates of $H_n$,
and Thm.~\ref{inter} yields $d(H_n, G_n) \leq 2 n$.
Set
\[
   K_n = \langle G, G^{\sigma_n}, \ldots, G^{\sigma_n^{n-1}} \rangle
       \leq G_n.
\]
Then $H_n \leq K_n \cong G^n$.

The character tables of $N$ and $G$ are as follows,
where the columns are indexed by the conjugacy classes
of the elements
$g_1 = ()$,
$g_2 = (1,3)(2,4)$,
$g_3 = (1,2)(3,4)$,
$g_3' = (1,2,3)$,
$g_4 = (1,4)(2,3)$,
$g_4' = (1,3)$,
$g_5 = (1,2,3,4)$.
\[
  \begin{array}{l|rrrrr}
          & g_1 & g_2 & g_3 & g_4  \\
    \hline
                                      \nu_1 & 1 &  1 &  1 &  1  \\
    \nu_2 & 1 &  1 &  -1 & -1  \\
    \nu_3 & 1 &  -1 & 1 &  -1  \\
    \nu_4 & 1 &  -1 & -1 & 1
  \end{array}
  \ \ \ \
  \begin{array}{l|rrrrr}
           & g_1 & g_2 & g_3' & g_4' & g_5 \\
    \hline
    \chi_1 & 1 &  1 &  1 &  1 &  1 \\
    \chi_2 & 1 &  1 &  1 & -1 & -1 \\
    \chi_3 & 2 &  2 & -1 &  0 &  0 \\
    \chi_4 & 3 & -1 &  0 &  1 & -1 \\
    \chi_5 & 3 & -1 &  0 & -1 &  1
  \end{array}
  \ \ \ \
  \begin{array}{rcl}
\chi_1|_N & = & \nu_1 \\
   \chi_2|_N & = & \nu_1 \\
   \chi_3|_N & = & 2 \nu_1  \\
   \chi_4|_N & = & \nu_2 + \nu_3 + \nu_4 \\
   \chi_5|_N & = & \nu_2 + \nu_3 +\nu_4
  \end{array}
\]

Set
\[
   X_n = \left\{ \chi_{i_1} \times \chi_{i_2} \times \cdots \times \chi_{i_n}
                 \in \Irr(K_n); i_1 \in \{ 4, 5 \},
                                i_j \in \{ 1, 2, 3 \} \mbox{\textrm{\ for\ }}
                 2 \leq j \leq n \right\}
\]
and
\[
   Y_n = \left\{ \chi^{G_n}; \chi \in X_n \right\}.
\]
Let $\Gamma_1$ be the undirected graph with vertex set $\{ 4, 5 \}$
and edge set $\{ \{ 4, 5 \} \}$,
$\Gamma_0$ be the undirected graph with vertex set $\{ 1, 2, 3 \}$
and edge set $\{ \{ 1, 3 \}, \{ 2, 3 \}, \{1,2\} \}$.
For $n \geq 2$, let $\Gamma_n$ be the Cartesian product of $\Gamma_1$
and $n-1$ copies of $\Gamma_0$,
that is,
$\Gamma_n$ has vertex set
\[
   \left\{ (i_1, i_2, \ldots, i_n);
          i_1 \in \{ 4, 5 \},
           i_j \in \{ 1, 2, 3 \} \mbox{\textrm{\ for\ }}
                 2 \leq j \leq n \right\},
\]
and there is an edge between $(i_1, i_2, \ldots, i_n)$ and
$(i_1', i_2', \ldots, i_n')$ if and only if there is a (unique) $j$
such that $i_k = i_k'$ for $k \not= j$ and $i_j \not= i_j'$
and either $\{ i_j, i_j' \} = \{ 4, 5 \}$
or $\{ i_j, i_j' \} \subset \{ 1, 2, 3 \}$.

\begin{lem}\label{lemma:graph}
\begin{itemize}
\item[(i)]
   $Y_n \subseteq \Irr(G_n)$,
   and mapping $\chi$ to $\chi^{G_n}$ defines a bijection from $X_n$ to $Y_n$.
\item[(ii)]
   For $\psi \in Y_n$ and $\psi' \in \Irr(G_n)$,
   if $\psi|_{H_n}$ and $\psi'|_{H_n}$ have a common constituent
   then $\psi' \in Y_n$.
\item[(iii)]
   Let $\psi = \chi^{G_n}$, $\psi' = (\chi')^{G_n}$ for $\chi, \chi' \in X_n$,
   with $\psi \not= \psi'$.
   Then $\psi|_{H_n}$ and $\psi'|_{H_n}$ have a common constituent
   if and only if there is an edge between $(i_1, i_2, \ldots, i_n)$ and
 $(i_1', i_2', \ldots, i_n')$ in $\Gamma_n$,
   where $\chi = \chi_{i_1} \times \chi_{i_2} \times \cdots \times \chi_{i_n}$
                                and
   $\chi' = \chi_{i_1'} \times \chi_{i_2'} \times \cdots \times \chi_{i_n'}$.
\item[(iv)]
   The distance of the vertices $(4, 1, 1, \ldots, 1)$ and
   $(4, 2, 2, \ldots, 2)$ of $\Gamma_n$ is $n-1$.
\item[(v)]
   The distance $d_{G_n}(\alpha_n, \omega_n)$ of the characters
   $\alpha_n:= \nu_2 \times \chi_1 \times \cdots \times \chi_1$
and\\  $\omega_n:= \nu_2 \times \chi_2 \times \cdots \times \chi_2$
   of $H_n$ is $n$.
\end{itemize}
\end{lem}
                           
\begin{proof}
Let $\psi = \chi^{G_n}$, where
$\chi = \chi_{i_1} \times \chi_{i_2} \times \cdots \times \chi_{i_n} \in X_n$,
that is, $\chi_{i_1}$ is faithful and the other $\chi_{i_j}$ are not.

For part~(i),
$\chi$ has inertia subgroup $K_n$ inside $G_n$.
Hence by Clifford's Theorem, see~\cite[Thm.~6.11]{Isa},
$\chi^{G_n}$ is irreducible.
The irreducible constituents of the restriction $\psi|_{K_n}$ are
the $n$ conjugates of $\chi$ by $\sigma_n$, i.~e.,
those characters where the $n$ components of $\chi$ are cyclically permuted.
Thus each constituent has exactly one faithful component.
Hence $\chi$ is the only constituent of $\psi|_{K_n}$ that lies in $X_n$.
Thus we get an inverse to the map $\chi\mapsto \chi^{G_n}$.

For part~(ii),
consider the restriction of the constituents of $\psi|_{K_n}$ to $H_n$.
We get irreducible constituents where the first component is a nontrivial
character of $N$
and all other components are non-faithful characters of $G$,
and irreducible constituents where the first component is the trivial
character of $N$ and exactly one other component is faithful.
Let $\psi' \in \Irr(G_n)$ have the property that $\psi'|_{H_n}$ and
$\psi|_{H_n}$ have a common irreducible constituent, which means that $0\ne(\psi|_{H_n},\psi'|_{H_n})=((\psi|_{H_n})^{G_n},\psi')$.
If this constituent is of the first kind then inducing it to $K_n$
yields a character with first component $\chi_4 + \chi_5$ and all other
components non-faithful.
If the common constituent is of the second kind then inducing it to $K_n$
yields a character with first component $\chi_1 + \chi_2 + 2 \chi_3$
and exactly one other component faithful. (Here we used that $(\mu \times \theta_2\times \cdots \times \theta_n)^{K_n}=(\mu^G\times \theta_2\times \cdots \times \theta_n)$, where $\theta_i\in \Irr(G)$, for $i=2\ldots n$, $\mu\in \Irr(N)$.)
 
In both cases, the irreducible constituents are cyclic shifts of characters in $X_n$,
thus inducing further from $K_n$ to $G_n$ yields characters
all whose irreducible constituents lie in $Y_n$.
Now note that $\psi'$ is one of them.

For part~(iii), note that
there is an edge between $(i_1, i_2, \ldots, i_n)$ and
$(i_1', i_2', \ldots, i_n')$ in $\Gamma_n$ if and only if
$\chi:=\chi_{i_1} \times \chi_{i_2} \times \cdots \times \chi_{i_n}$
and $\chi':=\chi_{i_1'} \times \chi_{i_2'} \times \cdots \times \chi_{i_n'}$
differ in exactly one component $\chi_{i_j}$, $\chi_{i_j'}$,
such that $\chi_{i_j}|_N$ and $\chi_{i_j'}|_N$ have a common constituent.
Let $\psi:=(\chi_{i_1} \times \chi_{i_2} \times \cdots \times \chi_{i_n})^{G_n}$,
and $\psi':=(\chi_{i_1'} \times \chi_{i_2'} \times \cdots \times \chi_{i_n'})^{G_n}$. Then $\psi|_{K_n}$ contains as a constituent $\chi_{i_1} \times \chi_{i_2} \times \cdots \times \chi_{i_n}$ and all its cyclic shifts,
$\psi'_{K_n}$ contains as a constituent $\chi_{i_1'} \times \chi_{i_2'} \times \cdots \times \chi_{i_n'}$ and all its cyclic shifts. When restricted further to $H_n$ the scalar product can be nonzero if and only if some of cyclic shifts of $\chi $
and some of cyclic shifts of $\chi'$ have  in the first component a restriction
that have a common component and all other components are equal. But then they must be shifted in the same way, since otherwise the faithful components were in different place.  Thus $\chi_{i_1} \times \chi_{i_2} \times \cdots \times \chi_{i_n}$
and $\chi_{i_1'} \times \chi_{i_2'} \times \cdots \times \chi_{i_n'}$ differ in
 exactly one component $\chi_{i_j}$, $\chi_{i_j'}$,
such that $\chi_{i_j}|_N$ and $\chi_{i_j'}|_N$ have a common constituent.

For part~(iv), observe that any shortest path from $(4, 1, \ldots, 1)$ to
$(4, 2, \ldots, 2)$ in $\Gamma_n$
replaces in each step exactly one $1$ by a $2$.

For part~(v), fix $n$ and let
$\alpha_n \sim_{G_n} \psi_1 \sim_{G_n} \psi_2 \sim_{G_n} \cdots
\sim_{G_n} \psi_m \sim_{G_n} \omega_n$
be a shortest path of related characters in $\Irr(H_n)$, of length $m+1$.
                  This means that there are irreducible characters
$\Phi_1, \Phi_2, \ldots, \Phi_{m+1}$ of $G_n$ such that
$\alpha_n$ and $\psi_1$ are constituents of $\Phi_1|_{H_n}$,
$\psi_i$ and $\psi_{i+1}$ are constituents of $\Phi_{i+1}|_{H_n}$,
for $1 \leq i \leq m-1$, and
$\psi_m$ and $\omega_n$ are constituents of $\Phi_{m+1}|_{H_n}$.
By Frobenius reciprocity we have  that $(\alpha_n^{G_n},\Phi_1)\ne 0$.
Since
$\alpha_n^{K_n} = (\chi_4 + \chi_5)\times\chi_1 \times\cdots\times\chi_1$
is a sum of characters in $X_n$, we know that $\Phi_1 \in Y_n$,
and part~(ii) implies that $\Phi_i \in Y_n$
for all $i \in \{ 1, 2, \ldots, m+1 \}$.
Let $\Theta_i$ be the unique character in $X_n$ with the property
$\Phi_i = \Theta_i^{G_n}$, for $1 \leq i \leq m+1$.
By part~(iii), $\Theta_i$ and $\Theta_{i+1}$ differ in at most one component.
Now $\Theta_1$ has $n-1$ components $\chi_1$,
and $\Theta_{m+1}$ has $n-1$ components $\chi_2$,
thus $m \geq n-1$ holds.
Conversely, any path of length $n - 1$ between
                                               $(4, 1, 1, \ldots, 1)$ and $(4, 2, 2, \ldots, 2)$ in $\Gamma_n$
yields a path of related characters from $\alpha_n$ to $\omega_n$,
of length $n$, hence $m+1 = n$.
\end{proof}

In order to prove that $d(H_n, G_n) = 2 n$,
it remains to show that  $d(H_n, G_n) \geq 2 n$ holds.
 If $d(H_n,G_n)\leq 2n-1=2(n-1)+1$, then
by Cor. \ref{distance}~(i) we have that every two irreducible characters of $H_n$ have distance  at most $n-1$.
However,
the characters $\alpha_n$ and $\omega_n$ constructed in
Lemma~\ref{lemma:graph} have  distance $n$, which is a contradiction. So we are done.


\begin{center}
 \textbf{Acknowledgments}
\end{center}

\noindent
The first author was supported by the
Stipendium Hungaricum PhD fellowship
at the Budapest University of Technology and Economics.
The second author gratefully acknowledges support by the German Research
Foundation (DFG) within the SFB-TRR 195
\textit{Symbolic Tools in Mathematics and their Applications}.
The third author was supported by the NKFI-Grants No.~115288 and 115799.

\bibliographystyle{amsalpha}

\providecommand{\bysame}{\leavevmode\hbox to3em{\hrulefill}\thinspace}
\providecommand{\MR}{\relax\ifhmode\unskip\space\fi MR }
\renewcommand{\MR}[1]{}
\providecommand{\MRhref}[2]{%
  \href{http://www.ams.org/mathscinet-getitem?mr=#1}{#2}
}
\providecommand{\href}[2]{#2}

\vfill
\noindent
Hayder Abbas Janabi,
Department of Algebra,
Budapest University of Technology and Economics,
H-1111 Budapest,
M\H{u}egyetem rkp.~3--9,
Hungary, \\
Department of Basic Sciences,
Faculty of Dentistry,
University of Kufa,
Najaf, Iraq,
e-mail: \texttt{haydera.janabi@uokufa.edu.iq} \\
\\
\noindent
Thomas Breuer,
Lehrstuhl f\"{u}r Algebra und Zahlentheorie,
RWTH Aachen University,
Pontdriesch 14--16,
D-52062 Aachen,
Germany,
e-mail: \texttt{sam@math.rwth-aachen.de} \\
\\
\noindent
Erzs\'ebet Horv\'ath,
Department of Algebra,
Budapest University of Technology and Economics,
H-1111 Budapest,
M\H{u}egyetem rkp.~3--9,
Hungary,
e-mail: \texttt{he@math.bme.hu}

\end{document}